\documentclass[11pt,leqno]{amsart}
\usepackage{amsfonts,amssymb,amsmath,amscd,amsthm}
\usepackage{hyperref}
\usepackage{mathrsfs, setspace}

\allowdisplaybreaks

\usepackage[english]{babel}

\newtheorem{thm}{Theorem}[subsection]

\newtheorem{cor}[thm]{Corollary}

\newtheorem{prop}[thm]{Proposition}

\theoremstyle{definition}

\newtheorem{rem}[thm]{Remark}

\newtheorem{exer}[thm]{Exercise}

\numberwithin{equation}{subsection}
\numberwithin{thm}{section}
\newcommand{\Hom}{\text{\rm Hom}}

\newcommand{\ext}{\operatorname{ext}}
\newcommand{\Ext}{\operatorname{Ext}}

\newcommand{\N}{{\mathbb N}}
\newcommand{\Q}{{\mathbb Q}}
\newcommand{\Z}{\mathbb Z}
\newcommand{\inv}{^{-1}}
\newcommand{\R}{\mathbb R}

\newcommand{\iso}{\xrightarrow{\sim}}

\newcommand{\C}{\mathbb C}

\newcommand{\fg}{{\mathfrak{g}}}

\newcommand{\tothe}{T^{\mu}_{\lambda}}
\newcommand{\fromthe}{T_{\mu}^{\lambda}}

\newcommand{\hd}{\operatorname{head}}

\newcommand{\rad}{\operatorname{rad}}
\newcommand{\stab}{\operatorname{Stab}}
\newcommand{\delr}{\Delta^{\mathrm{red}}}
\newcommand{\drr}{\Delta^\mathrm{red}_r}
\newcommand{\delrt}{\Delta^\mathrm{red}_2}
\newcommand{\nar}{\nabla_{\mathrm{red}}}
\newcommand{\nrr}{\nabla_\mathrm{red}^r}

\newcommand{\prdel}{\Delta^{p^r}}

\newcommand{\prnab}{\nabla_{p^r}}

\newcommand{\surj}{\twoheadrightarrow}
\newcommand{\inj}{\hookrightarrow}
\newcommand{\sco}{\mathscr O}

\newcommand{\pr}{\operatorname{pr}}
\newcommand{\gmd}{G\operatorname{-mod}}
\newcommand{\umd}{U_\zeta\operatorname{-mod}}
\newcommand{\w}{\widetilde}
\newcommand{\ch}{\operatorname{ch}}

\begin{document}

\title {On modules arising from quantum groups at $p^r$-th roots of unity}
\author{Hankyung Ko}
\address{Department of Mathematics \\
University of Virginia\\
Charlottesville, VA 22903} \email{hk2kn@virginia.edu}

\begin{abstract}
This paper studies the ``reduction mod $p$'' method, which constructs large classes of representations for a semisimple algebraic group $G$ 
from representations for the corresponding Lusztig quantum group $U_\zeta$ at a $p^r$-th root of unity. % of the same type as $G$.
%constructed via ``reducing mod $p$'' from the quantum group representations.
The $G$-modules arising in this way include the Weyl modules, the induced modules, and various reduced versions of these modules.
We present a relation between
$\Ext^n_G(V,W)$ and $\Ext^n_{U_\zeta}(V',W')$, when $V,W$ are obtained from $V',W'$ by reduction mod $p$. %(that give rise to the reduced modules.
Since the dimensions of $\Ext^n$-spaces for $U_\zeta$-modules are known in many cases, our result guarantees the existence of many new extension classes and homomorphisms between certain rational $G$-modules. One application is a new proof of James Franklin's result on certain homomorphisms between two Weyl modules.
We also provide some examples which show that the $p$-th root of unity case and a general $p^r$-th root of unity case are essentially different.
\end{abstract}

\maketitle

%\onehalfspacing

\section{introduction}
Let $G$ be a semisimple simply connected algebraic group, defined and split over a field $k$ of characteristic $p>2$. 
In studying the rational $G$-representations, the knowledge on the representations for the corresponding quantum group at a root of unity is very useful.
That is, consider $U_\zeta$, the Lusztig quantum group of the same type as $G$,
over a field $K$ of characteristic zero, specialized at a root of unity $\zeta\in K$, and ``compare'' $G$-mod, the category of rational $G$-representations, and $U_\zeta$-mod, the category of finite dimensional integrable $U_\zeta$-modules of type 1. 
%Parallel theories were developed for the two cases as explained in \cite{J,APW}. 
The two cases are closely related if the root of unity $\zeta$ has the order equals to $p$.
They have the same weight lattice $X$;
they are both highest weight categories with the poset $X^+$ of dominant weights (note, however, that $\gmd$ does not have projectives and has enough (infinite dimensional) injectives, while $\umd$ has both enough injectives and projectives);
the standard modules in two cases, indexed by the same highest weight $\gamma\in X^+$, have the same character $\chi(\gamma)$ given by Weyl's formula;
we have, in both cases, the linkage principle given by the affine Weyl group action;  
the translation functors between two orbits are defined in the same way and share similar properties in the two cases;
the standard modules in both cases have certain filtrations with a sum formula (Jantzen filtration);
the Frobenius kernel $G_1$ and the small quantum group $\mathbf u$ provide infinitesimal versions ($G_1T$-mod and $\mathbf uU_\zeta^0$-mod) of $\gmd$ and $\umd$.  
General theories including most of these can be found in \cite[II]{J} for the algebraic group representations and in \cite{APW} or \cite[II.H]{J} for the quantum group representations. 
%Both categories contain some distinguished modules which have strong homological properties.
In case $p$ is large enough, the infinitesimal versions of $G$-mod and $U_\zeta$-mod are even described using a common combinatorial category constructed in \cite{AJS}, which implies that the multiplicities of an irreducible module in a standard module (hence the irreducible characters) in the two cases are the same if the weights involved are small. 
%In general, the reduction mod $p$ method provides a more direct connection between the two cases. 

For general prime integer $p$, we have a better understanding on the quantum case than the algebraic group case. 
In particular, the Kazhdan-Lusztig polonomials compute the irreducible characters (given by the Lusztig character formula) and the dimensions of higher $\Ext$-spaces between two irreducibles. (See \cite{cps1} or \cite[II.C]{J} for the regular case and \cite{mine} for the singular case.)
In the algebraic group case, however, we need two significant restrictions in interpreting the Kazhdan-Lusztig polynomials.
The weights need be small, and the prime $p$ needs be large. 
A possible way to get rid of the restrictions is to consider the ``reduction mod $p$'' from the quantum case to the algebraic group case (see \S\ref{sssredmop}).
Parshall-Scott (\cite[Conjecture II]{PS14}) conjectured (under the Kazhdan-Lusztig equivalence) that the $\Ext$ computation in $\umd$ is valid in $\gmd$ if we replace the irreducible modules by appropriate reduction mod $p$ modules.
More precisely, the conjecture says
\begin{align}\label{con2}
\begin{split}
 \dim_K\Ext^n_{U_\zeta}(\Delta_\zeta(\gamma),L_\zeta(\gamma'))=&\dim_k\Ext^n_G(\Delta(\gamma),\nar(\gamma')),\\
 \dim_K\Ext^n_{U_\zeta}(L_\zeta(\gamma),\nabla_\zeta(\gamma'))=&\dim_k\Ext^n_G(\delr(\gamma),\nabla(\gamma')),\\
 \dim_K\Ext^n_{U_\zeta}(L_\zeta(\gamma),L_\zeta(\gamma'))  =&\dim_k\Ext^n_G(\delr(\gamma),\nar(\gamma')),
\end{split}
\end{align}
where $\gamma\in X^+$. Here $\delr(\gamma)$ (resp., $\nar(\gamma)$) is an abbreviation for $\delr_1(\gamma)$ (resp., $\nar^1(\gamma)$) which is defined in \S\ref{reducingmodules}.
The formula is shown in \cite{CPS7} to be valid for $p\gg 0$ in regular blocks.
Generally, one direction (``$\leq$'') of it follows from Proposition \ref{correducing}.

Let us now consider a $p^r$-th root of unity $\zeta$ instead of a $p$-th root of unity.
Everything makes sense since we can still reduce mod $p$ from the quantum side to $\gmd$. 
The use of $p^r$-th roots of unity, in fact, is not new and was studied by Lin in \cite{Lin}.%(The $p$-th root of unity case is introduced by Lusztig.)
We explain the reduction mod $p$ procedure in \S\ref{ssq}.
We show that we still have the inequality in \eqref{con2} with $r\geq 1$, (Proposition \ref{correducing}) while we don't have the equality for $r>1$ (\S\ref{a1exs}). 
As a consequence of Proposition \ref{correducing}, we obtain Franklin (\cite{Franmaps}) type of results on the maps between standard modules (corollaries \ref{franklin}, \ref{AKmaps}, \ref{multcor}).
We end the paper by checking that some other nice facts for the $r=1$ case becomes false for the $r>1$ cases (\S\ref{a1exs}).

\section*{Acknowledgement}
I am indebted to Brian Parshall and Leonard Scott for numerous discussions on the subject and helpful comments.

\section{Setting and notation}
\subsection{The algebraic group case}\label{ssalg}
 Let $G$ be a semisimple simply connected algebraic group over a field $k$ of characteristic $p>2$. We also assume that $G$ is defined and split over $\mathbb F_p\subset k$.
Fix a maximal (split) torus $T\subset G$ and a Borel subgroup $B\supset T$ in $G$. 
Let $R$ be the root system of $G$, $\Sigma$ be the set of simple roots, $R^+$ the set of positive roots, $X=X(T)$ be the set of weights, $X^+=X^+(T)$ be the set of dominant weights.
A general theory for the algebraic group $G$ and its representations is well explained in \cite[II]{J}.
 
Our interest is on $G$-mod, the category of rational $G$-modules. (We remind the reader that $\gmd$ is isomorphic to the category of of left comodules for the coordinate (Hopf) algebra $k[G]$ of $G$.)
The category $G$-mod is abelian, has enough injectives, and is a highest weight category in the sense of Cline-Parshall-Scott \cite{CPShwc} with the infinite poset $(X^+,\uparrow)$.
(See \S\ref{sslink} for the ordering $\uparrow$.)
For each $\gamma\in X^+$, we denote the standard object of highest weight $\gamma$ by $\Delta(\gamma)$ (which is the Weyl module, often denoted by $V(\gamma)$), the costandard object by $\nabla(\gamma)$ (which is the induced module and is denoted by $H^0(\gamma)$ in \cite{J}), and the simple object by $L(\gamma)$. 

 %See

\subsection{The quantum case and the integral case}\label{ssq}
Let $R$ be a semisimple simply connected (finite classical) root system as in \S\ref{ssalg}.
Let $(-,-)$ be a scalar product on $X(T)\otimes_\Z \R$ such that the smaller one among the integers $(\alpha,\alpha)$ for $\alpha\in R$ is $2$.
The quantum enveloping algebra $U_v=U_v(R)$ is 
defined over the function field $\Q(v)$ by generators 
$E_\alpha, F_\alpha, K_\alpha, K\inv_\alpha$ ($\alpha \in \Sigma$) 
and relations
	\begin{align*}
	K_\alpha K_\alpha =1 =K\inv_\alpha K_\alpha, &\ \ K_\alpha K_\beta=K_\beta K_\alpha,\\
	K_\alpha E_\beta K\inv_\alpha &=v^{(\alpha,\beta)}E_\beta,\\
	K_\alpha F_\beta K\inv_\alpha &=v^{-(\alpha,\beta)}F_\beta,\\
	E_\alpha F_\beta - F_\beta E_\alpha &=\delta_{\alpha\beta}\frac{K_\alpha-K\inv_\alpha}{v^{d_\alpha}-v^{-d_\alpha}}\\
	\displaystyle \sum_{s=0}^{1-a_{\alpha\beta}} (-1)^s {1-a_{\alpha\beta}\brack s}_\alpha &E^{1-a_{\alpha\beta}-s}_\alpha E_\beta E^s_\alpha =0,\\
	\displaystyle \sum_{s=0}^{1-a_{\alpha\beta}} (-1)^s {1-a_{\alpha\beta}\brack s}_\alpha &F^{1-a_{\alpha\beta}-s}_\alpha F_\beta F^s_\alpha =0,
	\end{align*}
where $d_\alpha=\frac{(\alpha,\alpha)}{2}$ (which can be $1,2$, or $3$) and $a_{\alpha\beta}=\langle \beta,\alpha^\vee\rangle=\frac{(\beta,\alpha)}{d_\alpha}$ for $\alpha, \beta \in \Sigma$.
Here ${1-a_{\alpha\beta}\brack s}_\alpha$ are the Gaussian binomial coefficients (see \cite[p.516]{J}).

 Letting $U^0_v$ denote the subalgebra of $U_v$ generated by the $K_\alpha^{\pm 1}$, $U_v^+$ the subalgebra generated by the $E_\alpha$, $U_v^-$ the subalgebra generated by the $F_\alpha$,
we have the triangular decomposition 
\begin{equation}\label{triandec}
U_v^-\otimes_{\Q(v)} U_v^0\otimes_{\Q(v)} U_v^+\xrightarrow[\mathrm{mult}]{\sim} U_v
\end{equation} induced by the muliplication in $U_v$.

 The algebra $U_v$ has an integral form $U_\mathscr{A}$ over $\mathscr{A}:=\Z [v,v\inv]$. It is defined as the $\mathscr{A}$-subalgebra of $U_v$ generated by the elements $K_\alpha^{\pm 1}$ and the $v$-divided powers $$E_\alpha^{(n)}=\frac{E_\alpha^n}{[n]^!_\alpha},\ \ \ \ \ \  F_\alpha^{(n)}=\frac{F_\alpha^n}{[n]^!_\alpha}$$ for $n>0$.
It is shown in \cite{Lusfdhf,Lusquanat1} that $U_\mathscr{A}$ also has a presentation by generators and relations compatible with the presentation for $U_v$. In particular, the triangular decomposition \eqref{triandec} restricts to $U_\mathscr{A}$, giving $$U_\mathscr{A}^-\otimes_\mathscr{A} U_\mathscr{A}^0\otimes_\mathscr{A} U_\mathscr{A}^+\iso U_\mathscr{A}.$$
Now, given a $\mathscr{A}$-algebra $\mathscr{B}$, one can take the tensor product to define the quantum group \[U_\mathscr{B}:=U_\mathscr{A}\otimes_\mathscr{A}\mathscr{B}.\]

Let $r$ be a positive integer and $\zeta\in \C$ be a primitive $p^r$-th root of unity.
Then specializing $v$ to $\zeta$ will give quantum algebras $U_\mathscr{A}\otimes\Z[\zeta]$ and $U_\mathscr{A}\otimes\Q(\zeta)$ at the root of unity $\zeta$.
For our purpose of relating the quantum group representations to $G$-mod, a modification on the base rings is necessary:
Instead of considering the above specializations,
%s $U_v\otimes_{v\mapsto\zeta}\Q(\zeta)$ and $U_\mathscr{A}\otimes_\mathscr{A}\Z[\zeta]$, 
we take the integral quantum algebras over a discrete valuation ring $\sco$ with the maximal ideal $(\pi)$, so that the residue field $\sco/(\pi)$ is isomorphic to a field $k$ of characteristic $p$ and the quotient field of $\sco$ is a field $K$ of characteristic zero. 
Such a triple ($K,\sco,k)$ is called a $p$-modular system.
For example, we take the localization $\sco:=\Z[\zeta]_{(\zeta-1)}$ in $\C$.
(If we don't want to start with taking some $\zeta\in\C$, we can set this up as follows. 
Consider the localization $\mathscr A_{(v-1,p)}$ and let $\sco=\mathscr A_{(v-1,p)}/(1+v+\cdots +v^{p^r-1})$. 
Then the image of $v$ in $\sco$ is a primitive $p^r$-th root of unity, which we rename to $\zeta$. %Note that the ideal $(\zeta-1)$ in $\sco$ contains $p$: Since \[0=\zeta^{p^r}-1=(\zeta^{p^{r-1}}-1)(1+\zeta^{p^{r-1}}+\cdots+\zeta^{(p-1)p^{r-1}})\] and $\zeta^{p^{r-1}}-1\neq 0$, we have \[1+\zeta^{p^{r-1}}+\cdots+\zeta^{(p-1)p^{r-1}}=0\]. But the image of the left hand side under $\zeta\mapsto 1$ is $p$. % maximal ideal is $(\zeta-1)$. )
In this case, the residue field $k$ of $\sco$ is the prime field $\mathbb F_p$, and the quotient field $K$ of $\sco$ is contained in $\C$.
%We can also make the $\sco$ larger if we want the residue field $k$ to be algebraically closed.

We denote by $U_\zeta$ the quantum group $U_v\otimes_{\Q(v)} K=(U_v\otimes_{v\mapsto\zeta}\Q(\zeta) )\otimes_{\Q(\zeta)} K$ thus obtained. 
The integral form $U_\sco$ will be denoted by $\w U_\zeta$.

 \subsubsection{The quantum case}\label{sssquantum}
 
The ``quantum case'' in this paper refers to $U_\zeta$-mod, the category of integrable finite dimensional $U_\zeta$-modules of type 1. 
Recall that a $U_\zeta$-module $M$ is called integrable if i) it is a weight module; ii) for each vector $v\in M$ $E_\alpha^{(n)}v=F_\alpha^{(n)}v=0$ for $n\gg0$. We say $M$ has type $1$ if the central elements $K_\alpha^{p^r}$ act on $M$ as an identity.
The weight lattice $X(U^0_\zeta)$ for $U_\zeta$-mod is identified with the weight lattice $X(T)$ of $G$, whence we write $X$ for this common weight lattice.
Then $\umd$ is a highest weight category with the poset $X^+$ of dominant weights.
We denote its standard objects by $\Delta_\zeta(\gamma)$, costandard objects by $\nabla_\zeta(\gamma)$, and simple objects by $L_\zeta(\gamma)$ for $\gamma\in X^+$.
Another important point is that $U_v$ and hence $U_\zeta$ are Hopf algebras.
Thus, $U_\zeta$-mod has a tensor product.
A general theory for $U_\zeta$-mod %, when $p$ is odd,
is developed in \cite{APW}. 
See also \cite{AndLink}.

\subsubsection{The integral case%: connecting the algebraic group case and the quantum case
}\label{ssinteg}

The integral version of the category $U_\zeta$-mod is $\w U_\zeta$-mod, the category of integrable finite (i.e., finitely generated over $\sco$) $\w U_\zeta$-modules of type 1. 
As in the algebraic group case and the quantum case, the highest weights of highest weight modules are indexed by the dominant weights.
Also as in the two cases, we have a tensor product of $\w U_\zeta$-modules using the Hopf algebra structure of $\w U_\zeta$.
%The category $\w U_\zeta$-mod is, in fact, an integral highest weight category with the poset $X^+$.

%Note that a $\w U_\zeta$-module is of type $1$ if and only if it is a $\w U_\zeta/(\{K_\alpha-1\}_{\alpha\in\Sigma})$-module.
%We can, thus, think of $\w U_\zeta$-mod as the category of integrable finitely generated $\w U_\zeta/(K_\alpha-1)$-modules. 

\subsubsection{Reduction mod $p$}\label{sssredmop}

We finally explain how the integral case provides a direct connection between the representation theory of $G$ and that of $U_\zeta$.
Recalling $$\w U_\zeta/(\{K_\alpha-1\}_{\alpha\in\Sigma})\otimes_\sco k\cong \text{Dist}(G)$$ from \cite{Lusquanat1},
we see that a module $\w M$ in $\w U_\zeta$-mod ``reduces mod $p$'' to a module $\w M\otimes_\sco k$ in $G$-mod.
(Note here that $K_\alpha$ acts as $1$ on the type $1$ module $\w M_k$.)
%We later discuss the ``reduction mod $p$'' procedure in detail.

\subsection{The linkage principle}\label{sslink}\label{ssdecomp}
We have identified the weight lattices for the algebraic group case, the quantum case, and the integral case.
This subsection defines the affine Weyl group action and linkage classes on the common weight lattice $X$.
Consider the $\R$-space $X\otimes_\Z\R$.
For $\alpha\in R$ and $m\in \Z$, denote by $s_{\alpha,m}$ the reflection with respect to the hyperplane in $X\otimes_\Z\R$ defined by the equation $\langle \lambda,\alpha^\vee\rangle=m.$ That is,
$$s_{\alpha,m}(\gamma)=\gamma-(\langle \gamma,\alpha^\vee\rangle-m)\alpha$$ for $\gamma\in X\otimes_\Z\R$. 
Let $W$ be the finite Weyl group of $R$. 
It is the reflection group generated by the simple reflections $s_\alpha=s_{\alpha,0}$:
\[W=\langle s_\alpha\ |\ \alpha\in \Sigma\rangle\]
For any $l\in\Z$, we define the affine Weyl group $W_l$ to be 
\[W_l=\langle s_{\alpha,ml}\ |\ \alpha\in R,\ m\in \Z\rangle\cong l\Z R\rtimes W.\]
\begin{rem}
The affine Weyl group in the quantum case (defined in \cite{AndLink}) is, in fact, slightly different. 
Let $l_\alpha:=\frac{l}{\gcd(l,d_\alpha)}$ for each $\alpha\in R$, where $d_\alpha=\frac{(\alpha,\alpha)}{2}$.
Then the affine Weyl group for the qunatum case is defined as
\[W_{D,l}=\langle s_{\alpha,ml_\alpha }\ |\ \alpha\in R,\ m\in \Z\rangle.\]
Since we assume that $l=p^r$ is odd, we have $s_{\alpha,ml_\alpha }=s_{\alpha,ml}$ except the case $p=3$ and $\alpha$ is a long root in type $G_2$.
We denote this affine Weyl group by $W_l$ (abusing notation in the $G_2$ situation above) in the paper.
See \cite[\S2.4.3]{hodge2016remarks} for a remark on this regarding the dual root system.
\end{rem}

Let $\rho$ be the sum of all fundamental weights, or equivalently, $\rho$ is the half sum of all positive roots. 
We almost always shift the action of $W_l$ on $X\otimes_Z\R$ by $\rho$, that is, $$w.\gamma=w(\gamma+\rho)-\rho$$ for $w\in W_l, \gamma\in X\otimes_Z\R$.

The standard (antidominant) $l$-alcove is by definition
\[^lC^-:=\{\gamma\in X\otimes_\Z\R\ |\ -l<\langle\gamma+\rho,\alpha^\vee\rangle <0 \textrm{ for all }\alpha\in R^+\}.\]
We call each $w.^lC^-$ an ($l$-)alcove. 
Another $l$-alcove we want to give a name is the bottom dominant alcove
\[^lC^+:=\{\gamma\in X\otimes_\Z\R\ |\ 0<\langle\gamma+\rho,\alpha^\vee\rangle <l \textrm{ for all }\alpha\in R^+\}.\]
 We call each set of the form 
\begin{align*}
F=\{\gamma\in X\otimes_\Z\R\ |\ l(n_\alpha-1) <&\langle\gamma+\rho,\alpha^\vee\rangle <ln_\alpha \textrm{ for all }\alpha\in R^+_0(F),\\
&\langle\gamma+\rho,\alpha^\vee\rangle =ln_\alpha\textrm{ for all }\alpha\in R^+_1(F)\}
\end{align*}
an ($l$-)facet, where $R^+=R^+_0(F)\sqcup R^+_1(F)$.
Then the closure $\overline{w.^lC^-}=w.\overline{^lC^-}$ (for any $w\in W_l$) is a union of facets and is a fundamental domain for the $W_l$-action.
Given $\gamma\in X\otimes_\Z \R$, there is a unique facet $F$ such that $\gamma$ is contained in the upper closure $\widehat{F}$, where we define the upper closure of $F$ as
\begin{align*}
\widehat{F}=\{\gamma\in X\otimes_\Z\R\ |\ l(n_\alpha-1) <&\langle\gamma+\rho,\alpha^\vee\rangle \leq ln_\alpha \textrm{ for all }\alpha\in R^+_0(F),\\
&\langle\gamma+\rho,\alpha^\vee\rangle =ln_\alpha\textrm{ for all }\alpha\in R^+_1(F)\}.
\end{align*}

We write a weight $\gamma$ (i.e., an element of $X$) as $w.\lambda$ for some $w\in W_l$ and a unique $\lambda$ in $\overline{^lC^-_\Z}:=\overline{^lC^-}\cap X$. (A more correct notation for this will be $\overline{^lC^-}_\Z$, but $\overline{^lC^-_\Z}$ looks better.)
We call a weight $\gamma=w.\lambda$ regular if $\lambda \in {^lC^-}$. 
We call $\gamma\in X$ singular if it is not regular. 
The choice of $w\in W_l$ is unique if and only if $\lambda$ is regular. If $\lambda$ is regular, this identifies $X^+\cap W_l.\lambda$ with the subset $$W_l^+:=\{w\in W_l\ |\ w.\lambda \in X^+\}$$ of $W_l$. For a general weight $\lambda$, we have preferred representatives. 
Recall that $W_l$ is generated by the subset $S_l$, which we choose to correspond to the simple reflections through the walls of $^lC^-$. 
Furthermore, ($W_l,S_l$) is a Coxeter system which has a natural ordering and a length function $l:W_l\to \Z$. 
Let $I:=\{s\in S_l\ |\ s.\lambda=\lambda\}$, $W_I=(W_l)_I:=\{ w\in W_l\ |\ w.\lambda=\lambda\}$, and let $W^I=(W_l)^I$ be the set of shortest coset representatives in $W_l/W_I$. 
Then for $w\in W_l^+$, we have $w\in W^I$ if and only if $w.\lambda\in\widehat{w.^lC^-}$. 
Now define $$W_l^+(\lambda):=W^{I}\cap W_l^+.$$ 
We identify $W_l^+(\lambda)$ with the set of dominant weights in the orbit of $\lambda$. 
The uparrow ordering of $X^+$ is defined to agree with the Coxeter ordering of $W_l$ (restricted to $W_l^+(\lambda)$) when restricted to $W_l^+(\lambda).\lambda\subset X^+$. 
(There is no order relation between two weights from two different $W_l$ orbits.)
See \cite[II.6, 8.22]{J} for more discussions on this.

%:$$\gamma\uparrow\gamma' \Leftrightarrow \gamma\leq s_1.\gamma \leq s_2s_1.\gamma\leq\cdots\leq s_n\cdots s_1.\gamma=\gamma' \textmd{ for some $s_i\in W_l$}$$

Now we consider the categories $\gmd$, and $\umd$. 
Take $l=p$ when we are in the algebraic group case. Take $l=p^r$ when we talk about the other cases.
By the linkage principle \cite[II.6]{J}, \cite[\S8]{APW}, the $G$-modules and $U_\zeta$-modules decomposes into the submodules (which are summands) whose composition factors have highest weights in the same $W_l$-orbits.
Using our notation, we can write this as the decomposition
$$\gmd=\bigoplus_{\mu\in\overline{^pC^-_\Z}}(\gmd)[W_l^+(\mu).\mu]$$
and
$$\umd=\bigoplus_{\mu\in\overline{^lC^-_\Z}}(\umd)[W_l^+(\mu).\mu].$$
(Given a highest weight category $\mathcal C$ with a poset $\Lambda$ and an ideal $\Gamma \unlhd\Lambda$, we set $\mathcal C[\Gamma]$ to be the Serre subcategory of $\mathcal C$ generated by the irreducibles in $\{L(\gamma)\}_{\gamma\in\Gamma}$. See \cite{CPShwc} for more details.)

\subsection{Weyl's character formula}
%We know the characters of standard and costandard modules in many module categories including $\gmd,\umd$ by Weyl's formula.
Consider the group algebra $\Z[X]$ of $X$. It has a basis $\{e(\gamma)\}_{\gamma\in X}$ with the multiplication $e(\gamma)e(\gamma')=e(\gamma+\gamma')$.
For $\gamma\in X$, the Weyl character 
\begin{equation}\label{weyl}
\chi(\gamma):=\frac{\sum_{w\in W}\det(w)e(w\gamma+w\rho)}{\sum_{w\in W}\det(w)e(w\rho)}=\frac{\sum_{w\in W}\det(w)e(w.\gamma)}{\sum_{w\in W}\det(w)e(w.0)}
\end{equation} is defined as an element in the fraction field of $\Z[X]$.
This element, while written as a fraction, actually belongs to $\Z[X]$.
We have for each $w\in W$ and $\gamma\in X$,
\begin{equation}\label{weylchreleq}
\chi(w\gamma)=\det(w)\chi(\gamma).
\end{equation}
If $\gamma\in X^+$, the formula \eqref{weyl} gives the characters of standard and costandard modules in many module categories, including $\gmd$ and $\umd$. That is, we have
\begin{equation}\label{weylcheq}
\ch\Delta(\gamma)=\ch\nabla(\gamma)=\ch\Delta_\zeta(\gamma)=\ch\nabla_\zeta(\gamma)=\chi(\gamma).
\end{equation}
 See, for example, \cite[II.5.10]{J}.

\section{Reducing modules modulo $p$}\label{reducingmodules}

We start with relating the standard and costandard modules in $\gmd$ and $\umd$. 
Recall that an $\sco$-submodule $\w M$ of a $U_\zeta$-module $M$ is called an admissible lattice if $\w M$ is $\sco$-free, $\w U_\zeta$-invariant and $K$-generates $M$ (i.e., $\w M\otimes_\sco K\cong M$).
In this case, the admissible lattice $\w M$ has a decomposition into weight $\sco$-free modules such that $\w M_\gamma\otimes_\sco K\cong M_\gamma$ for each $\gamma\in X$.
Choose a minimal admissible lattice $\w\Delta_\zeta(\gamma)$ in $\Delta_\zeta(\gamma)$. 
This is done simply by picking a highest weight vector $v$ in $\Delta_\zeta(\gamma)$ and letting $\w\Delta_\zeta(\gamma):=\w U_\zeta v$. For the costandard modules, we dualize this to take an admissible lattice $\w\nabla_\zeta(\gamma)$ in $\nabla_\zeta(\gamma)$ rather than dealing with the problem of what is a maximal lattice. 
Then we have
$$\w\Delta_\zeta(\gamma)_K\cong\Delta_\zeta(\gamma),\ \ \ \ \ \ \ \w\nabla_\zeta(\gamma)_K\cong\nabla_\zeta(\gamma).$$
and
$$\w\Delta_\zeta(\gamma)_k\cong\Delta(\gamma),\ \ \ \ \ \ \ \w\nabla_\zeta(\gamma)_k\cong\nabla(\gamma).$$
(Write $M_K:=M\otimes_\sco K$, $M_k:=M\otimes_\sco k$ for an $\sco$-module $M$.) 
So far we don't get any new $G$-modules. The irreducible $U_\zeta$-modules will give rise to the new modules of our interest. Let's do that.

Take a minimal admissible lattice $\w L_\zeta^{\mathrm{min}}(\gamma)$ in $L_\zeta(\gamma)$ and its dual $\w L_\zeta^{\mathrm{max}}(\gamma)$ in $L_\zeta(\gamma)$ (which is the dual of $L_\zeta(\gamma)$). 
Then we define $$\drr(\gamma):=(\w L_\zeta^{\mathrm{min}}(\gamma))_k,\ \ \ \ \ \ \ \ \nrr(\gamma):=(\w L_\zeta^{\mathrm{max}}(\gamma))_k.$$ 
These modules are not irreducible in general. In fact, they can be pretty big, as we see in the second sentence of the following observation.

\begin{prop}\label{trivsurj}
Let $\gamma\in X^+$. There is a surjective map $\Delta(\gamma)\to \drr(\gamma)$ for all $r\in \N$. It is an isomorphism if $\Delta_\zeta(\gamma)\cong L_\zeta(\gamma)$, in particular, if $\gamma\in \overline{^{p^r}\!C^+_\Z}$.
\end{prop}
\begin{proof}
We may assume that $\w L_\zeta^{\mathrm{min}}(\gamma)=\w U_\zeta\overline v$ where $\overline v$ is the image of a highest weight vector $v$ in $\Delta_\zeta(\gamma)$ under the surjective map $\Delta_\zeta(\gamma)\surj L_\zeta(\gamma)$. We also take $\w\Delta_\zeta(\gamma)=\w U_\zeta v$. Now the map 
$$\w\Delta_\zeta(\gamma)\surj\w L^{\mathrm{min}}(\gamma)$$ given by $v\mapsto \overline v$ induces the desired surjection 
$$\Delta(\gamma)\surj \drr(\gamma)$$ if we apply the exact functor $-\otimes_\sco k$. The second claim is trivial.
\end{proof}

\begin{cor}
For $\gamma\in X^+$, we have
\begin{align}
&\Delta(\gamma)\surj\drr(\gamma)\surj L(\gamma)\label{delsurjdrr}\\
&L(\gamma)\inj\nrr(\gamma)\inj\nabla(\gamma).
\end{align}

\end{cor}
\begin{proof}
By Proposition \ref{trivsurj} and its dual, it is enough to check that $\drr(\gamma)$ and $\nrr(\gamma)$ are not zero. But they arise from (nonzero) $\sco$-free lattices, hence cannot be zero.
\end{proof}

We now compare the reduction mod $p$ modules with another class of well-known modules.
Recall first the Steinberg tensor product theorems for the two cases:
\begin{equation}\label{stforG'}
L(\gamma_0+p^r\gamma_1)\cong L(\gamma_0)\otimes L(p^r\gamma_1)\cong L(\gamma_0)\otimes L(\gamma_1)^{[r]}
\end{equation}
\begin{equation}\label{stforU'}
L_\zeta(\gamma_0+p^r\gamma_1)\cong L_\zeta(\gamma_0)\otimes L_\zeta(p^r\gamma_1)\cong L_\zeta(\gamma_0)\otimes V(\gamma_1)^{[1]},
\end{equation}
where $\gamma_0\in X_r:=\{\gamma\in X^+\ |\ \langle \gamma
+\rho,\alpha^\vee\rangle <p^r,\ \ \forall \alpha\in \Pi\}$, $\gamma_1\in X^+$. %Both $-^{[1]}$ are what is called the Frobenius twists. 
We need to explain the terms:
The Frobenius twist $-^{[1]}$ for $G$ is equivalent to the map $f\mapsto f^p$ on the coordinate algebra $k[G]$. See \cite[I.9,II.3]{J}. 
The Frobenius twist $-^{[1]}$ for $U_\zeta$ (see \cite[II.H.6]{J}) is of a different nature because it is a map between $U_\zeta$ and $U(\fg)$, the universal enveloping algebra of the Lie algebra $\fg=\fg_K$.
Since the field $K$ is of characteristic zero, the Weyl module $V(\gamma_1)$ for $\fg$ is irreducible and has the Weyl character $\chi(\gamma_1)$.  
 %The Frobenius twist for $G$ is equivalent to the map $f\mapsto f^p$ on the coordinate algebra $k[G]$. (See \cite[I.9,II.3]{J}.) The Frobenius twist for $U_\zeta$ is explained in \cite[II.H.6]{J}.
We have the third Steinberg tensor product theorem, regarding reduction mod $p$ modules:

\begin{prop}\cite[Theorem 2.7]{Lin}%, \cite[Proposition 1.7]{CPS7}
\label{lin}
For $\gamma_0\in X_r$ and $\gamma_1\in X^+$, there are isomorphisms of $G$-modules
\begin{equation*}\drr(\gamma_0+p^r\gamma_1)\cong\drr(\gamma_0)\otimes\Delta(\gamma_1)^{[r]},\ \ \ \ 
\nrr(\gamma_0+p^r\gamma_1)\cong\nrr(\gamma_0)\otimes\nabla(\gamma_1)^{[r]}.\end{equation*}
\end{prop}

From now on, we always write $\gamma\in X^+$ as $\gamma=\gamma_0+p^r\gamma_1$ (uniquely) with $\gamma_0\in X_r$ and $\gamma_1\in X^+$.
Define $$\prdel(\gamma):=L(\gamma_0)\otimes\Delta(\gamma_1)^{[r]}$$ and $$\prnab(\gamma):=L(\gamma_0)\otimes\nabla(\gamma_1)^{[r]}.$$ 
Then Proposition \ref{lin} gives the following.

\begin{cor}\label{drrsurjpr}
For $\gamma\in X^+$, we have 
$$\drr(\gamma)\surj\prdel(\gamma),\ \ \ \ \ \prnab(\gamma)\inj \nrr(\gamma).$$
\end{cor}

The modules $\prdel(\gamma)$ for different $r\geq 1$ form a descending chain
$$\Delta^p(\gamma)\surj\Delta^{p^2}(\gamma)\surj\cdots\surj \prdel(\gamma)\surj\cdots.$$ This can be seen from the definition and the tensor product theorem \eqref{stforG'} as follows. Write \[\gamma=\gamma_0+p^{r-1}(\gamma_1+p\gamma_2)=(\gamma_0+p^{r-1}\gamma_1)+p^r\gamma_2\] with $\gamma_0\in X_{r-1}$, $\gamma_1\in X_1$, and $\gamma_2\in X^+$. Then we have
\begin{align*}
\Delta^{p^{r-1}}(\gamma)&=L(\gamma_0)\otimes\Delta(\gamma_1+p\gamma_2)^{[r-1]}\\
&\surj L(\gamma_0)\otimes (L(\gamma_1)\otimes\Delta(\gamma_2)^{[1]})^{[r-1]}\\
&\cong L(\gamma_0)\otimes L(\gamma_1)^{[r-1]}\otimes\Delta(\gamma_2)^{[r]}\\
&\cong 
L(\gamma_0+p^{r-1}\gamma_1)\otimes\Delta(\gamma_2)^{[r]}\\
&=\prdel(\gamma)
\end{align*}
for each $r>1$. The second line follows from combining \eqref{delsurjdrr} and Corollary \ref{drrsurjpr}. The fourth line follows from \eqref{stforG'}.

There is no obvious relation between the $\drr(\gamma)$ for $r\geq 1$. Instead, we know the characters of the modules $\drr(\gamma)$ in most (possibly all) cases since $\ch\drr(\gamma)=\ch L_\zeta(\gamma)$. % and we know the character $\ch L_\zeta(\gamma)$ in most cases. 
Suppose $l=p^r$ is a KL-good integer. By this we mean that the Kazhdan-Lusztig correspondence gives an equivalence between $\umd$ and a certain subcategory of the affine Lie category O associated to the root system $R$. (See \cite[\S7]{Tani} or \cite[\S2.2]{mine}. This is a very mild condition with no known non-KL-good integer in all types.)\label{klgood}
Recall the Kazhdan-Lusztig polynomial $P_{x,y}\in \Z[t,t\inv]$ defined for each pair $x,y\in W_{p^r}$ and let for $J\subset S_{p^r}$ and $y,w\in W^J$%, where $J=\{s\in S_l\ |\ s.\mu=\mu\}$,
\begin{equation}\label{paraklpoly}
P^J_{y,w}:=\sum_{x\in W_J}(-1)^{l(x)}P_{yx,w}.
\end{equation}
%This is called a \textit{parabolic Kazhdan-Lusztig polynomial} \cite{Deodpara,KTparabolic}.
For $\lambda\in \overline{^{p^r}\!C^-_\Z},w\in W_{p^r}^+$,
$$\ch\drr(\gamma)=\sum_{y\in W_{p^r}, y\in W_{p^r}^+(\lambda)} (-1)^{l(w)-l(y)}P^J_{y,w}(-1) \chi(y.\lambda)$$ where $J=\{s\in S_{p^r}\ |\ s.\lambda=\lambda\}$ and $\gamma=w.\lambda$.

We can use the characters to actually show that there is no map between $\drr(\gamma)$ and $\delr_{r'}(\gamma)$ in general for $r\neq r'\in \Z$. 
For $\chi,\chi'\in\Z[X]$, we say $\chi\geq\chi'$ %\Leftrightarrow 
if $\chi-\chi'\in \Z_{\geq 0}[X]$
and $\chi\not\geq\chi'$ otherwise.
Whether we require $r<r'$ or we require $r'<r$, we can easily find a case that 
$\ch \drr(\gamma)\not\geq\ch\delr_{r'}(\gamma)$. (See \S\ref{a1exs}.)
Since they have the same irreducible head $L(\gamma)$ and
$$[\drr(\gamma):L(\gamma)]=[\delr_{r'}(\gamma):L(\gamma)]=1,$$
any nonzero map between $\drr(\gamma)$ and $\delr_{r'}(\gamma)$ is surjective.
So $\ch \drr(\gamma)\not\geq\ch\delr_{r'}(\gamma)$ implies 
$$\Hom_G(\drr(\gamma),\delr_{r'}(\gamma))=0.$$

%Lusztig expected that the characters of the irreducible $G$-modules with small highest weights are also given by the Lusztig character for $p\geq h$. 
%This conjecture can be 

%This is not true for many $p\geq h$, while true for a large enough $p$. 

\section{Comparing the Jantzen sum formulas}

The Jantzen filtration on standard modules $\Delta(\gamma)\in\gmd$ is fully discussed in Jantzen's book \cite[II.8]{J}. 
An important consequence of the filtration is the Jantzen sum formula we state below. 
We need first to introduce a notation. Let $\nu_p$ be the $p$-adic valuation on $\Z$. That is, if $n\in \Z$ has the form $n=p^rd$ with $p,d$ relatively prime, then $\nu_p(n)=r$.
Recall the reflection $s_{\alpha,m}$ defined by %for each $\alpha\in R$ and $m\in \Z$
$$s_{\alpha,m}(\gamma)=\gamma-(\langle \gamma,\alpha^\vee\rangle-m)\alpha$$ for $\gamma\in X\otimes_\Z\R$.

\begin{prop}\cite[II.8.19]{J}\label{pJsum}
Let $\gamma\in X^+$. There is a filtration 
$$\Delta(\gamma)=V^0\supset V^1\supset\cdots $$
of $\Delta(\gamma)$ in $\gmd$ such that 
\begin{equation}\label{Jsum}
\sum_{i>0}\ch V^i=\sum_{\alpha\in R^+}\sum_{0<mp<\langle\gamma+\rho,\alpha^\vee\rangle} \nu_p(mp)\chi(s_{\alpha,mp}.\gamma)
\end{equation}
%where $\chi()$
and 
$$\Delta(\gamma)/V^1\cong L(\gamma).$$
\end{prop}
Let's denote the right hand side of \eqref{Jsum} by $\chi_J(\gamma)$, and put
$$\chi_J(\gamma,p^r):=\sum_{\alpha\in R^+}\sum_{0<mp^r<\langle\gamma+\rho,\alpha^\vee\rangle} \chi(s_{\alpha,mp^r}.\gamma).$$

We now rewrite the formula \eqref{Jsum} as

\begin{align}\begin{split}\label{Jsum'}
\chi_J(\gamma)&=\sum_{\alpha\in R^+}\sum_{0<mp<\langle\gamma+\rho,\alpha^\vee\rangle} \nu_p(mp)\chi(s_{\alpha,mp}.\gamma)\\
&=\sum_{\alpha\in R^+}(\sum_{0<mp<\langle\gamma+\rho,\alpha^\vee\rangle} \chi(s_{\alpha,mp}.\gamma)+\sum_{0<mp^2<\langle\gamma+\rho,\alpha^\vee\rangle} \chi(s_{\alpha,mp^2}.\gamma)
+\cdots)\\
&=\sum_r \chi_J(\gamma,p^r).
\end{split}\end{align}

The reason we do this is that
the Jantzen sum formula works for the quantum case, with a different formula:
\begin{prop}\cite[\S10]{APW}\label{qjsum}
Let $\gamma\in X^+$ and $r\geq 1$. There is a filtration of the $U_\zeta$-module $\Delta_\zeta(\gamma)$
$$\Delta_\zeta(\gamma)=V_\zeta^0\supset V_\zeta^1\supset\cdots $$
such that 
\begin{equation}\label{qJsum}
\sum_{i>0}\ch V_\zeta^i=\chi_J(\gamma,p^r)=\sum_{\alpha\in R^+}\sum_{0<mp^r<\langle\gamma+\rho,\alpha^\vee\rangle} \chi(s_{\alpha,mp^r}.\gamma)
\end{equation}
and 
$$\Delta_\zeta(\gamma)/V_\zeta^1\cong L_\zeta(\gamma).$$
\end{prop}

We can draw an observation from these formulas.
\begin{prop}\label{conseqjsum}
For $\gamma,\gamma'\in X^+$, we have
 $$[\Delta_\zeta(\gamma):L_\zeta(\gamma')]\neq 0\Rightarrow [\Delta(\gamma):L(\gamma')]\neq 0.$$
\end{prop}

\begin{proof}
Suppose $[\Delta_\zeta(\gamma):L_\zeta(\gamma')]\neq 0$. 
Since $[\Delta_\zeta(\gamma):L_\zeta(\gamma)]= [\Delta(\gamma):L(\gamma)]=1,$ we may assume that $\gamma>\gamma'$. 

By Proposition \ref{pJsum}, $[\Delta(\gamma):L(\gamma')]\neq 0$ if and only if $\ch L(\gamma')$ has a nonzero coefficient when we write $\chi_J(\gamma)$ as a ($\Z$-)linear combination of the characters of the irreducible $G$-modules (with non-negative coefficients).
In \eqref{Jsum'}, which says  $$\chi_J(\gamma)=\sum_r\chi_J(\gamma,p^r),$$ each $\chi_J(\gamma,p^r)$ is also a non-negative sum of irreducible $G$-characters by Proposition \ref{qjsum}, since all $U_\zeta$-characters are also $G$-characters. 
Thus, the claim is proved if we check that $\ch L(\gamma')$ has a nonzero coefficient when we write $\chi_J(\gamma,p^e)$ as a linear combination of the characters of the irreducible $G$-modules, where $e$ is such that $\zeta$ is a primitive $p^e$-th root of unity. 
By Proposition \ref {qjsum} and the assumption $[\Delta_\zeta(\gamma):L_\zeta(\gamma')]\neq 0,$ 
the character $\chi_J(\gamma,p^e)$ has a positive coefficient for $\ch L_\zeta(\gamma')$ when we write it as a non-negative sum of irreducible $U_\zeta$-characters. 
But $\ch L_\zeta(\gamma')$, when written as a sum of irreducible $G$-characters, has a nonzero $\ch L(\gamma')$ term.
\end{proof}

\begin{cor}\label{jscor}
For $\gamma,\gamma'\in X^+$ and $r\geq 1$, the module $\Delta(\gamma)$ has a composition factor $L(\gamma')$ if $\gamma'<\gamma$ are mirror images under the reflection through a wall of the $p^r$-facet containing $\gamma$.
\end{cor}
\begin{proof}
%The set $\{\ch L(\gamma)\}_{\gamma\in X^+}$ generates the space of all characters (of all finite $G$-modules). 
We take the integer $r$ as in the statement. 
That is, 
we have $\gamma'=s_{\beta,np^r}.\gamma \in X^+$ for an appropriate positive root $\beta$ and an integer $n$. 
Now observe in
$$\chi_J(\gamma,p^r)=\sum_{\alpha\in R^+}\sum_{0<mp^r<\langle\gamma+\rho,\alpha^\vee\rangle} \chi(s_{\alpha,mp^r}.\gamma)$$
that only $\chi(s_{\beta,np^r}.\gamma)$, among the Weyl characters appearing, has a nonzero $\ch L(s_{{\beta},np^r}.\gamma)$-term when it is written as a sum of irreducible $G$-characters. %Since $s_{\beta,np^r}.\gamma$ is dominant, 
Necessarily, the multiplicity $[\Delta_\zeta(\gamma):L_\zeta(\gamma')]$ is nonzero.
Proposition \ref{conseqjsum} gives the corollary.
\end{proof}

\section{Reducing morphisms modulo $p$}

We use the reduction mod $p$ procedure to construct many nontrivial elements in $\Hom$ and $\Ext^n$ spaces for $G$-mod.
%In particular, we very easily obtain some morphisms between standard modules similar to (and intersecting largely with) Franklin's result \cite{Franmaps}. 
\begin{prop}\label{reducingext}
Let $M,N\in \umd$ and $\w M, \w N\in \w U_\zeta\operatorname{-mod}$ be admissible lattices of $M, N$ respectively. Then for all $n\geq 0$, $$\dim_k\Ext^n_G(\w M_k,\w N_k)\geq\dim_K\Ext^n_{U_\zeta}(M,N).$$
\end{prop}
\begin{proof}
The short exact sequence $$0\to\w N\xrightarrow{\pi}\w N\to \w N_k\to 0$$ of $\w U_\zeta$-modules induces the long exact sequence 
\begin{align}\begin{split}\label{les}
0\to&\Hom_{\w U_\zeta}(\w M,\w N)\xrightarrow{\pi}\Hom_{\w U_\zeta}(\w M,\w N)\to\Hom_{\w U_\zeta}(\w M,\w N_k)\to\cdots\\
\to&\Ext^n_{\w U_\zeta}(\w M,\w N)\xrightarrow{\pi}\Ext^n_{\w U_\zeta}(\w M,\w N)\to\Ext^n_{\w U_\zeta}(\w M, \w N_k)\to\\
\to&\Ext^{n+1}_{\w U_\zeta}(\w M,\w N)\xrightarrow{\pi}\Ext^{n+1}_{\w U_\zeta}(\w M,\w N)\to\cdots
\end{split}
\end{align}
of $\sco$-modules, where the map $\pi$ is multiplication by the generator of the maximal ideal of $\sco$.

We also have $$\Ext_{\w U_\zeta}^n(\w M,\w N)\otimes_\sco K\cong\Ext_{U_\zeta}^n( M,N),$$ by \cite[(2.9), Theorem 3.2]{DS}. 
Let $d_n$ be the ($K$-)dimension of this space. Thus, $\Ext_{\w U_\zeta}^n(\w M,\w N)=\sco^{\oplus d_n}\oplus T_n$, where $T_n$ is a torsion $\sco$-module. 
%Note that $T_0=0$.(yes?)
So the sequence \eqref{les} is of the form
\begin{align}\begin{split}\
%0&\to \sco^{\oplus d_0} \xrightarrow{\pi}\sco^{\oplus d_0}\to\Hom_{\w U_\zeta}(\w M,\w N_k)\\
%&\to\cdots\\
\cdots&\to\sco^{\oplus d_n}\oplus T_n\xrightarrow{\pi}\sco^{\oplus d_n}\oplus T_n\to\Ext^n_{\w U_\zeta}(\w M, \w N_k)\to\\
&\to\sco^{\oplus d_{n+1}}\oplus T_{n+1}\xrightarrow{\pi}\sco^{\oplus d_{n+1}}\oplus T_{n+1}\to\cdots.
\end{split}
\end{align}
Since $\sco/\pi\sco\cong k$, we have 
$$\Ext^n_{\w U_\zeta}(\w M, \w N_k)\cong k^{\oplus d_n}\oplus T_n/\pi T_n\oplus \ker (T_{n+1}\xrightarrow{\pi}T_{n+1}).$$
The proof is complete using \cite[(2.9), Theorem 3.2]{DS}, which says
$$\Ext^n_{\w U_\zeta}(\w M,\w N_k)\cong\Ext^n_G(\w M_k,\w N_k).$$ 
\end{proof}

The following is an immediate consequence.
\begin{prop}\label{correducing}
Let $\gamma,\gamma'\in X^+$. We have
\begin{enumerate}
\item $\dim_k\Ext^n_G(\Delta(\gamma),\Delta(\gamma'))\geq\dim_K\Ext^n_{U_\zeta}(\Delta_\zeta(\gamma),\Delta_\zeta(\gamma'))$;
%\item $\dim_k\Ext^n_G(\nabla(\gamma),\nabla(\gamma'))\geq\dim_K\Ext^n_{U_\zeta}(\nabla_\zeta(\gamma),\nabla_\zeta(\gamma'))$;
\item $\dim_k\Ext^n_G(\Delta(\gamma),\drr(\gamma'))\geq\dim_K\Ext^n_{U_\zeta}(\Delta_\zeta(\gamma),L_\zeta(\gamma'))$;
\item $\dim_k\Ext^n_G(\Delta(\gamma),\nrr(\gamma'))\geq\dim_K\Ext^n_{U_\zeta}(\Delta_\zeta(\gamma),L_\zeta(\gamma'))$;
\item $\dim_k\Ext^n_G(\drr(\gamma),\nrr(\gamma'))\geq\dim_K\Ext^n_{U_\zeta}(L_\zeta(\gamma),L_\zeta(\gamma'))$;
\item $\dim_k\Ext^n_G(\drr(\gamma),\drr(\gamma'))\geq\dim_K\Ext^n_{U_\zeta}(L_\zeta(\gamma),L_\zeta(\gamma'))$;
\end{enumerate}
and similar inequalities for the dual modules (replace ``$\Delta$'' by ``$\nabla$''). 
\end{prop}

The right hand sides of Proposition \ref{correducing} (2)-(5) are known in most cases by the following result. 

\begin{prop}\cite[Theorem 4.10, Theorem 4.14]{mine}\cite[Conjecture III]{PS14}
Suppose $p^r$ is KL-good for the root system $R$ (see page~\pageref{klgood}). Let $\mu\in \overline{^{p^r}\!C^-_\Z}$, $J=\{s\in S_{p^r}\ |\ s.\mu=\mu\}$, and $y,w\in W^+(\mu)$. 
We have $$\sum_{n=0}^\infty \dim \Ext_{U_\zeta}^n(\Delta_\zeta(y.\mu),L_\zeta( w.\mu))t^n=t^{l( w)-l( y)}P^{J}_{y,w}(t\inv)$$ 
and 
$$\sum_{n=0}^\infty \dim \Ext_{U_\zeta}^n(L_\zeta( y.\mu),L_\zeta( w.\mu))t^n=
\sum_{ z\in W^+(\mu)}t^{l( y)+l( w)-2l( z)}P^{J}_{ z, y}(t\inv)P^{J}_{ z, w}(t\inv)$$
where the polynomial $P^J_{y,w}$ is as in \eqref{paraklpoly}.
\end{prop}

The ``$r\geq 1$''-analogues of \cite[Conjecture II]{PS14} would say that the ``$\geq$'' are ``$=$'' for Proposition \ref{correducing} (3), (4). 
We give some examples in \S\ref{a1exs} below where this is a strict inequality for $r>1$. %In case \cite[Conjecture II]{ps14} is true, then the left hand side of () is greater than

Unlike in the other inequalities in Proposition \ref{correducing}, the right hand side in Proposition \ref{correducing} (1) is not known in general. 
(A result on $\Ext$ between two (co)standard modules in special cases can be found in \cite{aparkergoofydim}.)
Another difference between this case and the rest is that the left hand side in Proposition \ref{correducing} (1) does not depend on $r$. 
Considering all the cases $r\geq 1$ together, we obtain
\begin{align}\begin{split}\label{eqdeldel}
\dim_k\Ext^n_G(\Delta(\gamma),\Delta(\gamma'))
\geq\max\{\dim_K\Ext^n_{U_\zeta}(\Delta_\zeta(\gamma),\Delta_\zeta(\gamma'))\ |\ \zeta^{p^r}=1, r\geq 1\}
\end{split}\end{align}

We explore some special cases where we can say something about the dimensions of $\Ext^n_{U_\zeta}(\Delta_\zeta(\gamma),\Delta_\zeta(\gamma'))$ for the rest of this subsection.

It will be convenient to employ the following convention when writing weights. Recall that we identify the weight lattices for $G$ and for $U_\zeta$, for any root of unity $\zeta$. Now write a $G$-weight $\gamma$ as $\gamma=w.\lambda$ where $\lambda\in\overline{{^{p^r}\!C}^-}\cap X$ and $w\in W^+_{p^r}(\lambda)\subset W_{p^r}\subset W_p$. The following two corollaries intersect largely with Franklin's result \cite{Franmaps}.
%are also proved in Franklin \cite{Franmaps}.
%See also Remark \ref{Franrem}.% for more precise comparison between the corollaries and \cite{Franmaps}.

\begin{cor}\label{franklin}
Let $r\geq 1$ be such that $p^r\geq h$. Let $\mu\in\overline{{^{p^r}\!C}^-}\cap X$, $w\in W_{p^r}^+\subset W_p$ and $s\in S_{p^r}\subset W_p$ (So $s$ may not be in $S_p$). If $ws>w$, then $$\Hom_G(\Delta(w.\mu),\Delta(ws.\mu))\neq 0.$$
In other words, there is a nonzero map $$\Delta(\gamma)\to \Delta(\gamma')$$ if $\gamma'>\gamma\in X^+$ and $\gamma'$ is the reflection image of $\gamma$ through a wall of the $p^r$-facet containing $\gamma$. %This map is unique up to scalar.
\end{cor}
\begin{proof}
First consider the case where $\mu\in {{^{p^r}\!C}^-}\cap X$ is regular.
The condition on $p^r$ ensures the existence of (a regular and) a subregular weight for $U_\zeta$ with $\zeta$ a primitive $p^r$-th root of unity. (See \cite[II.6.3]{J}.)
We can, thus, apply a translation argument \cite[II.7.19]{J} to obtain
\begin{equation*}\label{quanfr}
\Hom_{U_\zeta}(\Delta_\zeta(w.\mu),\Delta_\zeta(ws.\mu))\cong K
\end{equation*}
 in the quantum case.
The corollary follows from Proposition \ref{correducing} (1).

Now we treat the general weight $\mu\in\overline{{^{p^r}\!C}^-}\cap X$. 
Again, by Proposition \ref{correducing} (1), it is enough to obtain
$$\Hom_{U_\zeta}(\Delta_\zeta(w.\mu),\Delta_\zeta(ws.\mu))\neq 0.$$
Pick a regular weight $\lambda\in {{^{p^r}\!C}^-}\cap X$ (possible since $p^r\geq h$) and consider the translation functor $\tothe$ in $U_\zeta$-mod.
We may assume that $w\in W^J$.
We have 
\begin{align*}
\Hom_{U_\zeta}(\Delta_\zeta(w.\mu),\Delta_\zeta(ws.\mu))&\cong\Hom_{U_\zeta}(\tothe\Delta_\zeta(w.\lambda),\tothe\Delta_\zeta(ws.\lambda))\\
&\cong\Hom_{U_\zeta}(\fromthe\tothe\Delta_\zeta(w.\lambda),\Delta_\zeta(ws.\lambda))
\end{align*}
But the surjection $$\fromthe\tothe\Delta_\zeta(w.\lambda)\surj \Delta_\zeta(w.\lambda)$$
induces an inclusion $$\Hom_{U_\zeta}(\Delta_\zeta(w.\lambda),\Delta_\zeta(ws.\lambda))\inj\Hom_{U_\zeta}(\fromthe\tothe\Delta_\zeta(w.\lambda),\Delta_\zeta(ws.\lambda)).$$ 
The left hand side is nonzero by the regular case done in the first paragraph.
\end{proof}

\begin{rem}
Since $\hd\Delta(\gamma)\cong L(\gamma)$, if there is a nonzero map from
$\Delta(\gamma)$ to $\Delta(\gamma')$ then $L(\gamma)$ is a composition factor of $\Delta(\gamma')$. Thus, Corollary \ref{franklin} implies Corollary \ref{jscor}.
The same remark applies to Corollary \ref{AKmaps} below.
\end{rem}

In fact, we know many more morphisms between standard modules in the quantum case, from which we can reduce mod $p$.
\begin{cor}\label{AKmaps}
Let $r\geq 1$ be such that $p^r\geq h$.
Let $\lambda\in\overline{{^{p^r}\!C}^-}\cap X$, $s\in S_{p^r}\setminus I$, where $I=\{s\in S_{p^r}\ |\ s.\lambda=\lambda\}$
For $w\in W_{p^r}^{(S_{p^r}\setminus\{s\})}\cap W_{p^r}^+$ and $x<y\in W_{(S_{p^r}\setminus\{s\})}(=(W_{p^r})_{(S_{p^r}\setminus\{s\})})$, we have
$$\Hom_G(\Delta(wx.\lambda),\Delta(wy.\lambda))\neq 0.$$
\end{cor}
\begin{proof}
A nonzero morphism in the quantum case 
$$\Delta_\zeta(wx.\lambda)\to\Delta_\zeta(wy.\lambda)$$
when $\lambda$ is regular is well known or explained in \cite[Remark 3.6, Proposition 3.7]{AK}.
It is a composition of maps obtained in Corollary \ref{franklin} and also uses translation functors in showing that it is nonzero.
The proof of Corollary \ref{franklin} gives the singular case. 
Finally, use Proposition \ref{correducing}.
\end{proof}

%\begin{rem}\label{Franrem}
%\cite{Franmaps} has a condition involving $\gcd(d_\alpha,p)$
%\end{rem}

A very similar proof gives the next corollary.

\begin{cor}\label{multcor}
In the situation of Corollary \ref{franklin}, we have $$\Ext^1_G(\Delta(w.\mu),\Delta(ws.\mu))\neq 0.$$
\end{cor}
\begin{proof}
By Proposition \ref{correducing} (1), it is enough to show that $\Ext_{ U_\zeta}^1(\Delta_\zeta(w.\mu),\Delta_\zeta(ws.\mu))\neq 0.$

The regular case is done by the argument in \cite[II.7.19]{J}. For the singular case, take a regular weight $\lambda$ in ${^{p^r}\!C}^-$ and consider the translation functor $\tothe$ of $U_\zeta$-modules.
Then,
\begin{align*}
\Ext^1_{U_\zeta}(\Delta_\zeta(w.\mu),\Delta_\zeta(ws.\mu))%&\cong\Ext^1_{U_\zeta}(\tothe\Delta_\zeta(w.\lambda),\tothe\Delta_\zeta(ws.\lambda))\\
&\cong\Ext^1_{U_\zeta}(\fromthe\tothe\Delta_\zeta(w.\lambda),\Delta_\zeta(ws.\lambda))
\end{align*}
Assuming $w\in W^J$, there is a short exact sequence
$$0\to M\to\fromthe\tothe\Delta_\zeta(w.\lambda)\to\Delta_\zeta(w.\lambda)\to 0,$$
where $M$ has a $\Delta$-filtration with sections $\Delta_\zeta(wx.\lambda)$, $x\in W^J\setminus\{e\}$. 
Taking $\Hom_{U_\zeta}(-,\Delta_\zeta(ws.\lambda))$, we have (a part of) a long exact sequence
\begin{align}\begin{split}\label{Mles}
\to\Hom_{U_\zeta}(M,\Delta_\zeta(ws.\lambda))
\to\Ext_{U_\zeta}^1(\Delta_\zeta(w.\lambda),\Delta_\zeta(ws.\lambda))\to\Ext_{U_\zeta}^1(\fromthe\tothe\Delta_\zeta(w.\lambda),\Delta_\zeta(ws.\lambda))\to
.\end{split}\end{align}
Since the head of $M$ is a direct sum of irreducibles of highest weight $wt$ with $t\in J$, and since $wt\not<ws$ for $t\neq s \in S_{p^r}$, the first term $\Hom_{U_\zeta}(M,\Delta_\zeta(ws.\lambda))$ in \eqref{Mles} is zero.
The second term $\Ext_{U_\zeta}^1(\Delta_\zeta(w.\lambda),\Delta_\zeta(ws.\lambda))$ in \eqref{Mles} is nonzero by the regular case considered above.
Therefore $\Ext_{U_\zeta}^1(\fromthe\tothe\Delta_\zeta(w.\lambda),\Delta_\zeta(ws.\lambda))$ is also not zero, which proves the claim.
\end{proof}

\begin{rem}
The condition $p^r\geq h$ in Corollary \ref{franklin}, \ref{AKmaps}, \ref{multcor} can be replaced by the KL-good condition by transferring the assertions in $\umd$ to the affine case via the Kazhdan-Lusztig correspondence, using the combinatorics of Fiebig \cite{Fiecomb} to move the level (which is proportional to the root order), and then transferring the assertion back to the quantum case where the assertion is proved. %See the proof of \cite[Theorem 4.10]{mine} where we do this.
\end{rem}

\section{Type $A_1$ examples}\label{a1exs}
Let $G=SL_2$. This section provides some examples of the $r=2$ case which shows that many nice results for the $r=1$ case do not generalize to $r>1$. We list the $r\geq 1$ versions of the ``nice results'' that are disproved in this section. 

\begin{enumerate}
\item If $p\gg 0$ and $\lambda\in {^pC^-_\Z}$, then $\Delta(w.\lambda)$ for each $w\in W^+_p$ has a filtration with sections of the form $\drr(\gamma)$. (The $r=1$ case is proved in \cite{PS13}.)

\item  If $p\gg 0$ and $\lambda\in {^pC^-_\Z}$, %then $\drr(w.\lambda)$ for $w\in W_p^+$ has left parity with respect to all irreducible $G$-modules, or equivalently, $\drr(w.\lambda)\in \E^L$, or equivalently, 
then we have 
\begin{equation}\label{2con2}
\Ext^n_G(\drr(w.\lambda),\nabla(y.\lambda))=0\ \text{  if } l(w)- l(y)\not\equiv n \text{ mod}\ 2.
\end{equation}
(The $r=1$ case is proved in \cite{CPS7}.%: See \S\ref{ps2}.
)
\end{enumerate}

We take $p=3$ to have concrete numbers, % rather than expressions involving ``$p$''.
but all the examples here work for a larger $p$. Note that the condition ``$p\gg 0$'' means, by our convention, ``the Lusztig conjecture for $G$ is true and $p\geq 2h-2$''. Thus, $3\gg 0$ for $SL_2$.
%Also, the reader will be able to convince himself that 

In this case, the dominant weights are identified with the integers $n\in\Z_{\geq 0}$.
The Jantzen region in this notation is defined by the condition $n\leq 8$. 
On the quantum side, we have $U_\zeta=U_\zeta(\mathfrak{sl_2})$ with $\zeta$ a primitive $9$th root of unity and $U_{\zeta^3}$ the corresponding quantum group at a $3$rd root of unity. 

Let's consider the regular orbit containing $0\in {^p C^-_\Z}$. The highest weights are $0,4,6,10,12$, $16,18,\cdots$. 
We express the radical filtration of a $G$-module via the following notation. 
\[ M= \begin{array}{c}
 M/\rad M=\hd M \\
 \rad M/\rad^2 M \\ 
 \rad^2 M/\rad^3 M\\
 \cdots \end{array}\]

It is easy to check that 
\[ \Delta(6)=\delrt(6)= \begin{array}{c}
 L(6) \\
 L(4) 
 \end{array}\]
and
 \[\delrt(10)= \begin{array}{c}
 L(10) ,
 \end{array}\]
while the structure of $\Delta(10)$ is either

\begin{equation}\label{case1}
\Delta(10)= \begin{array}{c}
 L(10) \\
 L(4) \\ 
 L(6)
 \end{array}\tag{Case 1}
\end{equation}

or 

\begin{equation}
\Delta(10)= \begin{array}{c}
 L(10) \\
 L(4) \oplus L(6)
 \end{array}\tag{Case 2}.
\end{equation}
(We cannot have 
\[\Delta(10)=\begin{array}{c}
L(10)\\
L(6)\\
L(4)
\end{array}\]
because $\Delta(10)\surj \delr(10)=\begin{array}{c}
L(10) \\
L(4)
\end{array}.$)

In either case, $\Delta(10)$ does not have a filtration with sections of the form $\delrt(\gamma)$.
We actually know which is the case. 
Suppose Case 2 is true. Then, we have \[\Ext_G^1(L(10),L(6))\neq 0.\]
Since the weight $10$ is the only one out of the Jantzen region among the weights $0,4,6,10$, 
we have \[\Ext_G^1(L(a),L(b))\neq 0\] 
for every pair of weights $a,b\in\{0,4,6,10\}$ in two adjacent alcoves.
Applying \cite[Theorem 5.3]{cps1} to the category $(\gmd)[0,4,6,10]$, we have $\ch L(10)=\ch L_\zeta(10)$.
This contradicts $\ch L_\zeta(10)=\ch L(10)+\ch L(4)$.
Thus, \eqref{case1} is the case.

%This can in face be seen just from the character formulas for any type. 

Now we check that $\delrt(10)$ does not satisfy \eqref{2con2}.
Consider, for each case, the following sequence of distinguished triangles in $D^b(\gmd)$.

\newtheorem*{c1}{Case 1}
\newtheorem*{c2}{Case 2}

\[\Delta(10)\to Y_0=\delrt(10)\to Y_1= \begin{array}{c}
 L(4) \\ 
 L(6)
 \end{array}[1]\to,\]

\[\Delta(6)[1]\to Y_1= \begin{array}{c}
 L(4) \\ 
 L(6)
 \end{array}[1]\to Y_2=L(4)[1]\oplus L(4)[2]\to,\]
 
\[\Delta(4)[1]\oplus\Delta(4)[2]\to Y_2=L(4)[1]\oplus L(4)[2]\to Y_4= L(0)[2]\oplus L(0)[3]\to,\] 

$$\Delta(0)[2]\oplus\Delta(0)[3]\xrightarrow{\cong}Y_4= L(0)[2]\oplus L(0)[3]\to 0\to,$$

 using 
\[\Delta(0)=L(0),\ \ \ \ \ \ \ \ \Delta(4)= \begin{array}{c}
 L(4) \\ 
 L(0)
 \end{array}.\] 
%The two cases give basically the same sequence.
%We see that the ``wrong'' shifts $\Delta(4)[1]$ and $\Delta(0)[2]$ appear. 
Since we have
\begin{equation}\label{ortho}
\ext^n(\Delta(\lambda),\nabla(\lambda'))=
\begin{cases} k\ \ \ \ \lambda=\lambda', n=0\\
0\ \ \ \  \text{otherwise}\end{cases},
\end{equation}
 the sequence above of distinguished triangles show that $\Ext_G^i(\delrt(10),\nabla(4))$ has dimension one at $i=1,2$ and that $\Ext_G^i(\delrt(10),\nabla(0))$ has dimension one at $i=2,3$. (See \cite{cps1} or \cite[\S3.2]{mine} for details.) In particular, the $p^2$-analogue of \cite[Conjecture II]{PS14} is not true.

%\subsection{Translating the second }\label{sscon2}

\bibliographystyle{abbrv}\def\cprime{$'$} \def\cprime{$'$} \def\cprime{$'$} \def\cprime{$'$}

\end{document}